\documentclass[smallextended]{svjour3}

\usepackage[intlimits]{amsmath}
\usepackage[english]{babel}
\usepackage{amsfonts,amssymb}
\usepackage{verbatim}

\usepackage[left = 3cm, right = 1.5cm, bottom = 2cm, top = 2cm]{geometry}

\newcommand{\Rb}{\mathbb R}
\newcommand{\Hb}{\mathbb H}
\newcommand{\Zb}{\mathbb Z}
\newcommand{\lb}{\left(}
\newcommand{\rb}{\right)}
\newcommand{\enbrace}[1]{\lb #1 \rb}
\newcommand{\enbracesq}[1]{\left[ #1 \right]}

\newcommand{\setdef}[2]{\left\{ #1\ \left|\ #2 \right.\right\}}

\newcommand{\vmod}[1]{\left| #1 \right|}
\newcommand{\vvec}[1]{ {\rm\bf #1} }
\newcommand{\Vc}{ \vvec{c} }
\newcommand{\Ve}{ \vvec{e} }

\renewcommand{\leq}{ \leqslant }
\renewcommand{\geq}{ \geqslant }

\spnewtheorem{statement}{Statement}{\bf}{\it}

\DeclareMathOperator{\dist}{dist}

\author{Dmitry Todorov}
\title{\bf \LARGE Non-normability of spaces of Keplerian orbits}

\institute{D. Todorov \at
Chebyshev laboratory, Saint-Petersburg State University, \\
14th line of Vasiljevsky Island, 29B,  199178, \\
Saint-Petersburg, Russia, \\
\email{todorovdi@gmail.com} }

\date{Received: 21.10.2012}

\thanks{Research supported by the Chebyshev laboratory
(grant of the Russian government N 11.G34.31.0026).}

\begin{document}
\maketitle

\begin{abstract}
We prove that spaces of Keplerian curvilinear orbits, all orbits and
elliptic orbits with marked pericenter cannot carry a norm, compatible
with their standard topology. We also prove that the space of Keplerian
elliptic orbits without marked pericenter cannot carry a norm, compatible
 with the natural metrics on it.
\end{abstract}

\makeatletter{}\section{Introduction}

In \cite{msko} and \cite{holmetr} several spaces of Keplerian orbits are defined.
The metric spaces of curvilinear orbits $\Hb(b)$ and of all orbits $\Hb$ are described in \cite{msko}. Metrics on both spaces agree with the standard topology, induced from $\Rb^d,$ where $d$ is either $6$ or $7,$ correspondingly. Moreover it is shown that both $\Hb(b)$ and $\Hb$ are $5$-dimensional algebraic manifolds without singularities. They are open and arcwise-connected subsets of the corresponding $\Rb^d.$ The space $\Hb$ is complete. Spaces $E^*$ and $E$ of elliptic orbits were introduced in \cite{holmetr}.
All these spaces are important and often used. One can hope to introduce a structure of a normed space, that agree with the standard topology, on one of these spaces.
We show that one shouldn't try to do this, because it is impossible due to some topological reasons.
\makeatletter{}\section{Definitions}
In the space $\Rb^3\setminus\{0\}$ consider the following equation
\begin{equation}
\ddot{\vvec{r}}+\varkappa^2 \frac{ \vvec{r} }{r^3}=0. \label{orbitode}
\end{equation}

\begin{remark}
Let $\vvec{r}(t)$ be a nonextendable solution of equation \eqref{orbitode}, $t_\alpha<t<t_\omega.$ In case of a general position$ -t_\alpha = t_\omega = \infty.$ For linear trajectories for nonnegative full energy $h,$ one of the boundary values $t_\alpha, t_\omega$ is finite and other is infinite. For $h<0$ both boundary values $t_\alpha, t_\omega$ are finite.
\end{remark}

\begin{definition}
We define an \emph{orbit} as a class of all parametrizations $\vvec{r}(t(t')),$ where $t(t')$ is a continuous increasing on $(t'_\alpha, t'_\omega )$ function and $t'_\alpha = \lim_{t'\to t'_\alpha} t',\ t'_\omega = \lim_{t'\to t'_\omega} t'.$
\end{definition}

\begin{remark}
Roughly speaking, an orbit is a set of points $\{\vvec{r}(t)\}$ with a fixed ordering (orientation, that is defined by the time of reaching a point).
\end{remark}

It is easy to see that the space of all orbits is $5$-dimensional. This is due to $6$ integration constants of equation \eqref{orbitode} when we do not consider a position of a point on an orbit. As usual, it is easier to embed this space in a space, having higher dimension. We do it below.

Let $h$ be the full energy, $\vvec{c}$ be the angular momentum vector, $\vvec{e}$ be the Laplace vector:
\begin{eqnarray}
h=\frac{1}{2}\dot{\vvec{r}}-\frac{\varkappa^2}{2},\quad
\vvec{c}=\vvec{r} \times \dot{\vvec{r}},\quad
\vvec{e} = \frac{\dot{\vvec{r}} \times \vvec{c}}{\varkappa^2} - \frac{\vvec{r}}{r}. \label{ints}
\end{eqnarray}
Variables \eqref{ints} are uniquely determined by an orbit. Seven (scalar) variables \eqref{ints} satisfy the following two relations
\begin{eqnarray}
\Vc\cdot\Ve=0, \label{Hconecond} \\
2hc^2-\varkappa^4(e^2-1)=0. \label{Hquadriccond}
\end{eqnarray}
One can find these formulae in every elementary book on celestial mechanics. Recall that a curvilinear orbit is uniquely determined by orthogonal vectors $\vvec{c},\vvec{e}$ when $c>0.$
In particular any vector $\vvec{c}\neq 0$ corresponds to a unique circular orbit, for which $\vvec{e}=0.$ Sometimes it is convinient to distinguish between circular orbits for fixed  $\vvec{c}\neq 0.$
To be able to do it, one can take a nonzero vector $\vvec{e}$ and pass to a limit when $e\to 0$ saving the direction of the vector $\vvec{e}.$ We call them orbits \emph{with marked pericenter}.

For a linear orbit we have $\vvec{c}=0,\ e=1.$ This gives us a two-dimensional manifold (in fact it is the unit sphere, since any unit vector $\vvec{e}$ is allowed). Nevertheless, any vector $\vvec{e}$ correspond to a continuum of orbits with parameter $h\in (-\infty,\infty).$

Now we are able to define spaces of orbits.

Fix a nonnegative $b.$

\begin{definition}
We define the \emph{set of curvilinear orbits} with angular momentum, greater than $b$ as
\begin{equation*}
\Hb(b)=\setdef{(\Vc,\Ve) \in \Rb^6}{\hbox{ $\Vc,\Ve$ satisfy  \eqref{Hconecond} and $c>b$ }}.
\end{equation*}
\end{definition}

\begin{remark}
The space $\Hb(0)$ contains all curvilinear orbits.
\end{remark}

\begin{definition}
We define the \emph{set of all orbits} as:
\begin{equation*}
\Hb=\setdef{(\Vc,\Ve,h) \in \Rb^7}{\hbox{ $\Vc,\Ve$ satisfy \eqref{Hconecond} and $\Vc,\Ve,h$ satisfy \eqref{Hquadriccond} } }.
\end{equation*}
\end{definition}

On $\Hb(b)$ and $\Hb$ we consider metrics that are induced from $\Rb^6$ and $\Rb^7$ correspondingly.

\begin{definition}
We define the \emph{set of all elliptic orbits with marked pericenter} as
\begin{equation*}
E^*=\setdef{(\Vc,\Ve) \in \Hb(0)}{e < 1}.
\end{equation*}
Consider the following metrics on it
\begin{equation*}
\varrho^*_p(\mathcal{E}_1,\mathcal{E}_2)=
\begin{cases}
\displaystyle \enbrace{ \frac{1}{2\pi} \int_{u} \enbracesq{ \dist(Q_1(u),Q_2(u)) }^p  du  }^{1/p}, & p<\infty; \\ \max\limits_u\dist(Q_1(u),Q_2(u)), & p=\infty.
\end{cases}
\end{equation*}
Here $\mathcal{E}_k\in E^*,$ and $Q_k$ is a parametrization that corresponds to $\mathcal{E}_k.$ The parameter $u$ defines a position on the elliptic orbit. It is the eccentric anomaly.

\end{definition}

\begin{remark}
All the metrics $\varrho_p^*$ generate the same topology.
\end{remark}

Define the following pseudometrics on $E^*$
\begin{equation*}
\varrho_p(\mathcal{E_1} , \mathcal{E_2})=
\begin{cases}
\displaystyle \min\limits_s\enbrace{ \frac{1}{2\pi} \int_{u}  \enbracesq{\dist(Q_1(u),Q_2(u+s)) }^p du }^{1/p}  , & p<\infty; \\ \min\limits_s \max\limits_u\dist(Q_1(u),Q_2(u+s)), & p=\infty.
\end{cases}
\end{equation*}
Here again $\mathcal{E}_k\in E^*,$ and $Q_k$ is a parametrization that corresponds to $\mathcal{E}_k.$ The parameter $u$ defines a position on the elliptic orbit. It is the eccentric anomaly.

\begin{definition}
We define \emph{the set of all elliptic orbits without marked pericenter} as
\begin{equation*}
E=E^*/\sim,
\end{equation*}
where $x\sim y$ when $\varrho_p(x,y)=0.$
\end{definition}

\begin{remark}
In the previous definition all the points $(\Vc,0)$ become one point. On the space $E$ pseudometrics $\varrho_p$ are nondegenerate. This means that they are metrics on $E.$
\end{remark}
\makeatletter{}\section{Main Results}

\begin{theorem}  \label{thm:mainHb}
For any $b\geq 0$ there does not exist a way to define a structure of a finite-dimensional topological vector space over $\Rb$ on the space $\Hb(b).$
\end{theorem}

\begin{theorem}  \label{thm:mainH}
There does not exist a way to define a structure of a finite-dimensional topological vector space over $\Rb$ on the space $\Hb.$
\end{theorem}

\begin{theorem}  \label{thm:mainEstar}
There does not exist a way to define a structure of a finite-dimensional topological vector space over $\Rb$ on the space $E^*.$
\end{theorem}

\begin{theorem}  \label{thm:mainE}
For any $p\in[1,\infty],$ there does not exist a way to define a structure of a finite-dimensional normed vector space over $\Rb$ on the space $E,$ that agree with metric $\varrho_p.$
\end{theorem}

\begin{remark}
Words ``there does not exist a way to define a structure of a finite-dimensional topological vector space'' on some space mean that there does not exist a finite-dimensional topological vector space, that is homomorphic to this space. Analogously, words ``there is not exist a way to define a structure of a finite-dimensional normed vector space'' on some space mean that there isn't exist a finite-dimensional topological vector space with a metric compatible with the topology, that is isometric to this space.
\end{remark}

If it is impossible to define a structure of a finite-dimensional topological vector space then it is also impossible to define a structure of a finite-dimensional normed vector space.
\makeatletter{}\section{Proofs}

The proofs are based on the following topological fact. Topological vector spaces have very simple structure from the topological point of view. All topological facts we are going to use can be found in \cite{hatcher}.

Recall that all the spaces  $\Hb(b),$  $\Hb,$ $E^*$ and $E$ are of dimension $5.$ If one can define a linear structure that agree with the topology, then topological dimension and linear dimension (dimension of finite-dimensional vector space over $\Rb$) must coincide. This is due to the fact that a finite-dimensional vector space has topological dimension that coincides with linear dimension and that manifolds of different topological dimensions cannot be homeomorphic.

We will also use the following two well-known facts:

\begin{statement} \label{st:homeofindim}
Any two finite-dimensional topological vector spaces of the same dimension are homeomorphic to each other.
\end{statement}

\begin{statement} \label{st:isometricfindim}
All norms on a finite-dimensional normed space are equivalent.

\end{statement}

\makeatletter{}\subsection{The space of curvilinear orbits $\Hb(b)$}

\begin{statement} \label{st:curvcong}
For any number $b\geq 0$ we have $\Hb(b)\cong TS^2 \times \Rb.$
\end{statement}
\begin{proof} Recall that
$$\setdef{\Vc\in\Rb^3}{c>b} = \bigcup_{r>b} rS^2 = S^2 \times (b,+\infty).$$
For a fixed vector $\Vc$ equation \eqref{Hconecond} defines a plane in the space $\Rb^3,$ that is orthogonal to the vector $\Vc.$ Considering such planes for vectors from the sphere of a fixed radius $r,$ we get the tangent bundle
$$T(rS^2)=\setdef{(v_1,v_2,v_3,v_4,v_5)\in\Rb^5}{(v_1,v_2,v_3)\in rS^2,\ (v_4,v_5)\in \Rb^2}.$$
Since we can take any $r,$ we have $\Hb(b) = \bigcup\limits_{r>b} T(rS^2)\subset \Rb^5.$ The space $\bigcup\limits_{r>b} T(rS^2)$ is homeomorphically mapped to $TS^2 \times (b,+\infty)$ by the following mapping:
$$f(\vvec{v})=\enbrace{\frac{(v_1,v_2,v_3)}{\vmod{(v_1,v_2,v_3)}},
(v_4,v_5), \vmod{(v_1,v_2,v_3)}},$$
where $\vmod{\cdot}$ is a standard Euclidean norm.

\end{proof}

\begin{statement} \label{st:curvhomnoneq}
The spaces $\Hb(b)$ and $\Rb^5$ are not homotopy equivalent to each other.
\end{statement}
\begin{proof}
$TS^2$ is homotopy equivalent $S^2,$ bacause any fiber bundle with contractible layers is homotopy equivalent to its base. Therefor the space $\Hb(b)$ is homotopy equivalent to $S^2 \times \Rb,$ that is homotopy equivalent to the sphere $S^2.$ It is well-known that $S^2$ is not homotopy equivalent to $\Rb^5,$ because the latter has the property that any $2$-sphere in it is homotopic to a point (one can use a linear homotopy, for example), but the sphere $S^2$ does not have this property. To be more precise, here we use the fact that spaces that are homotopy equivalent must have the same homotopy groups. In our case $\pi_2(S^2)=\Zb,$ but all homotopy groups of $\Rb^5$ are trivial.
\end{proof}

\begin{corollary}[Theorem \ref{thm:mainHb}] \label{cor:curvenonhomeo}
For any $b\geq 0$ there does not exist a way to define a structure of a finite-dimensional topological vector space over $\Rb$ on the space $\Hb(b).$
\end{corollary}
\begin{proof}
Homeomorphic spaces must be homotopy equivalent, thus it is enough to use the previous statement.
\end{proof}

\begin{remark} \label{rem:curvenorm}

One can prove that there does not exist a way to define a structure of a finite-dimensional normed vector space over $\Rb$ on the space $\Hb(b),$ that agree with standard metric using a simpler argument. Any finite-dimensional normed space is complete, while $\Hb(b)$ is not.
\end{remark}

\subsection{The space of all orbits $\Hb$}

\begin{statement}
The space $\Rb^5$ can not be split by a compact set into two disjoint connected subsets, each having compact closure. I.e. the complement of each compact set in $\Rb^5$ has exactly one connected component with noncompact closure.
\end{statement}
\begin{proof}
Every compact subset $M$ of a metric space is bounded, therefor there exists a ball $B$ with radius big enough, that contain $M.$ Thus every connected component of the complement of $M$ is either inside the ball $B$, or contain the complement of the ball $B,$ that is connected. This is the reason, why this component can be only one, since it is the biggest connected subset, containing the complement. Then the components inside the ball are bounded and their closures are compact and one that contain the complement of the ball is unbounded and has noncompact closure.
\end{proof}

\begin{statement} [Theorem \ref{thm:mainH}]
The space $\Hb$ is not homeomorphic to $\Rb^5.$
\end{statement}
\begin{proof}
For $h=-1$ equation \eqref{Hquadriccond} defines an ellipsoid, which is compact. In the space $\Hb$ both components of the complement of this set ($h<-1$ and $h>-1;$ obviously, they are both connected) are unbounded, since for every $h$ a solution of the equations \eqref{Hquadriccond} and \eqref{Hconecond} exists (for example, one can take $\Vc=0,\ \Ve\in S^2$).
\end{proof}

\begin{remark}
The space $\Hb$ is complete and in this case one cannot just repeat the proof of remark \ref{rem:curvenorm}.
\end{remark}

\begin{statement}
$\Hb\cong \Rb \times S^2 \times S^2.$
\end{statement}
\begin{proof}
The linear change of the coordinates $\vvec{c} \mapsto (\varkappa^2/2)\vvec{c}$ homeomorphically transforms equation \eqref{Hquadriccond} into the following
$$
e^2-hc^2=1.
$$
This change of coordinates does not affect equation \eqref{Hconecond}. We show how to construct a map $E_1:\Hb\to (0,\infty)\times S^2 \times S^2$ and its inverse. Let
\begin{eqnarray}
\vvec{u}=\vvec{c}+\vvec{e},\ \vvec{v}=\vvec{c}-\vvec{e}, \label{uvdef} \\
\vvec{p}=\frac{\vvec{u}}{u},\ \vvec{q}=\frac{\vvec{v}}{v}, \label{pqdef} \\
s=\begin{cases} - \displaystyle\frac{e^2}{c^2}, & c\neq 0, \\ -\infty, & c=0, \end{cases} \label{sdef}\\
k=\begin{cases} e^{-h}-e^s, & c\neq 0, \\ e^{-h}, & c=0. \end{cases} \label{kdef}
\end{eqnarray}
Define a value of $E_1(h,\vvec{c},\vvec{e})$ as $(k,\vvec{p},\vvec{q}),$ where $k,p$ and $q$ are defined using equalities \eqref{uvdef}-\eqref{kdef}.

We construct the inverse mapping $E_2$ in the following way: for $\vvec{p},\vvec{q}\in S^2\subset\Rb^3$ we take
\begin{eqnarray*}
a=1-\vvec{p}\cdot\vvec{q},\ b=1+\vvec{p}\cdot\vvec{q}, \\
s= \begin{cases} -\displaystyle\frac{b}{a}, & \vvec{p}\neq \vvec{q}, \\ -\infty, & \vvec{p} = \vvec{q}. \end{cases}
\end{eqnarray*}
After that we find $h$ from equation \eqref{kdef} and set
$$
r=\sqrt{\frac{2}{b-ah}}.
$$
Now, it is enough to find $\vvec{c}$ and $\vvec{e}$ from equations \eqref{uvdef} for $\vvec{u}=r\vvec{p},\ \vvec{v}=r\vvec{q}$ and set $E_2(k,\vvec{p},\vvec{q})=(h,\vvec{c},\vvec{e}).$

\end{proof}

\makeatletter{}\subsection{The spaces $E^*$ and $E$}

\begin{statement}
For any $p\in[1,\infty]$ the topology, generated on $E^*$ by the metric $\varrho^*_p,$ coincides with the topology, induced from $\Rb^6.$
\end{statement}
\begin{proof}

It is enough to verify that every Euclidean ball $B_{\dist}$ contain some $\varrho^*_p$-ball $B_{\varrho^*_p}$ and vice versa.

Continuous dependence on the initial data of the solution of differential equation \eqref{orbitode} imply that any ball $B_{\dist}$ contains some ball $B_{\varrho^*_p}.$

If one writes inequality $\varrho^*_p\leq \varrho^*_\infty$ and a Taylor polynomial for inequality (13) from \cite{holmetr} then it is possible to estimate the size of  $\vmod{\Vc_1-\Vc_2}$ and $\vmod{\Ve_1-\Ve_2}$ for small $\varrho^*_\infty((\Vc_1,\Ve_1),(\Vc_2,\Ve_2)).$ This implies that any ball $B_{\varrho^*_p}$ contains some ball $B_{\dist}.$

\end{proof}

Denote $$BT(S^2)=\setdef{(p,v)\in TS^2}{\vmod{v} < 1}.$$

\begin{statement} [Theorem \ref{thm:mainEstar}]
There does not exist a way to define a structure of a finite-dimensional topological vector space over $\Rb$ on the space $E^*.$
\end{statement}
\begin{proof}
Analogous to the proof of statement \ref{st:curvcong} one can show that $E^*\cong BT(S^2) \times (0,\infty).$ Using exactly the same reasoning as in the proof of statement \ref{st:curvhomnoneq} we obtain that the space $E^*$ is not homotopy equivalent to the space $\Rb^5$ and thus, cannot be homeomorphic to it, like in corollary \ref{cor:curvenonhomeo}.
\end{proof}

\begin{remark}
As in remark \ref{rem:curvenorm} the absence of normability of the space $E^*$ can be proved easier, using the fact that this space is not complete.
\end{remark}

\begin{statement}
The space $E$ is not complete.
\end{statement}
\begin{proof}
Fix a plane and consider inside it a sequence of circles with radii (bigger semi-major axes) tending to zero. This is a fundamental sequence in any metric $\varrho_p$, and it it does not have a limit since $E$ does not contain ellipses with zero bigger half-axis.
\end{proof}

\begin{corollary}[Theorem \ref{thm:mainE}]
For any $p\in[1,\infty],$ there does not exist a way to define a structure of a finite-dimensional normed vector space over $\Rb$ on the space $E,$ that agree with metric $\varrho_p.$
\end{corollary}
\begin{proof}
Analogous to the proof of \ref{rem:curvenorm}.
\end{proof}

\makeatletter{}\section{Acknowledgements}

Author is grateful to professor K. Kholshevnikov for offering a problem and S. Podkorytov for fruitful discussions.

This work was supported by Chebyshev laboratory (grant of the Russian government N 11.G34.31.0026).


\begin{thebibliography}{99}
\bibitem{msko} Kholshevnikov, K.V.: Metric Spaces of Keplerian Orbits. Celest. Mech. Dyn. Astron. 100(3), 169-179 (2008)
\bibitem{holmetr}  Kholshevnikov, K.V., Vassilev N.N.: Natural Metrics in the Spaces of Elliptic Orbits. Celest. Mech. Dyn. Astron. 89(2), 119-125 (2004)
\bibitem{hatcher} Hatcher A.: Algebraic topology. Cambridge University Press, Cambridge (2002)
\end{thebibliography}
\end{document}